\documentclass[a4paper, 12pt]{amsart}
\usepackage{amsmath, amssymb, eucal, amscd, amstext, enumerate}

\newtheorem{thm}{Theorem}[section]
\newtheorem{lem}[thm]{Lemma}
\newtheorem{eg}[thm]{Example}
\newtheorem{prop}[thm]{Proposition}
\newtheorem{cor}[thm]{Corollary}
\newtheorem{defn}[thm]{Definition}
\newtheorem{rem}[thm]{Remark}
\newtheorem{ntn}[thm]{Notation}

\newenvironment{prf}{{\noindent \textbf{Proof:}\ }}{\hfill $\Box$\\ \smallskip}

\numberwithin{equation}{section}

\newcommand{\id}{{\rm id}}
\newcommand{\ti}{\tilde}
\newcommand{\smnoind}{\smallskip\noindent}

\newcommand{\FZ}{\mathfrak{Z}}
\newcommand{\FU}{\mathfrak{U}}
\newcommand{\FN}{\mathfrak{N}}
\newcommand{\CF}{\mathcal{F}}

\newcommand{\CN}{\mathcal{N}}
\newcommand{\CU}{\mathcal{U}}
\newcommand{\bl}{\left}
\newcommand{\br}{\right}

\newcommand{\supp}{\operatorname{supp}}

\begin{document}
\title[Automatic continuity and $C_0(\Omega)$-linearity]{Automatic continuity and $C_0(\Omega)$-linearity of linear maps between $C_0(\Omega)$-modules}

\author{Chi-Wai Leung, Chi-Keung Ng  \and Ngai-Ching Wong }

\address[Chi-Wai Leung]{Department of Mathematics, The Chinese
University of Hong Kong, Hong Kong.}
\email{cwleung@math.cuhk.edu.hk}

\address[Chi-Keung Ng]{Chern Institute of Mathematics and LPMC, Nankai University, Tianjin 300071, China.}
\email{ckng@nankai.edu.cn}

\address[Ngai-Ching Wong]{Department of Applied Mathematics, National Sun Yat-sen University, Kaohsiung, 80424, Taiwan.}
\email{wong@math.nsysu.edu.tw}

\thanks{The authors are supported by National Natural Science Foundation of China (10771106) and Taiwan NSC grant (NSC96-2115-M-110-004-MY3).}

\date{\today}

\begin{abstract}
Let $\Omega$ be a locally compact Hausdorff space. We show that any
local $\mathbb{C}$-linear map (where ``local'' is a weaker notion than $C_0(\Omega)$-linearity) between Banach $C_0(\Omega)$-modules
are ``nearly $C_0(\Omega)$-linear'' and ``nearly bounded''. As an
application, a local $\mathbb{C}$-linear map $\theta$ between
Hilbert $C_0(\Omega)$-modules is automatically $C_0(\Omega)$-linear.
If, in addition, $\Omega$ contains no isolated point, then any $C_0(\Omega)$-linear map between Hilbert $C_0(\Omega)$-modules is
automatically bounded.
Another application is that if a sequence of maps
$\{\theta_n\}$ between two Banach spaces ``preserve
$c_0$-sequences'' (or ``preserve ultra-$c_0$-sequences''), then $\theta_n$ is bounded for large enough $n$
and they have a common bound. Moreover, we will show that if
$\theta$ is a bijective ``biseparating'' linear map from a ``full''
essential Banach $C_0(\Omega)$-module $E$ into a ``full'' Hilbert
$C_0(\Delta)$-module $F$ (where $\Delta$ is another locally compact
Hausdorff space), then $\theta$ is ``nearly bounded'' (in fact, it
is automatically bounded if $\Delta$ or $\Omega$ contains no isolated point)
and there exists a homeomorphism $\sigma: \Delta \rightarrow \Omega$
such that $\theta(e\cdot \varphi) = \theta(e)\cdot \varphi\circ
\sigma$ ($e\in E, \varphi\in C_0(\Omega)$).

\end{abstract}

\maketitle

\section{Introduction}

\medskip

A linear map $\theta$ between the spaces of continuous sections of two bundle spaces over the same
locally compact Hausdorff base space $\Omega$ is said to be \emph{local} if  for any continuous section $f$, one has $\supp \theta(f) \subseteq \supp f$,
or equivalently, for each $g\in C_0(\Omega)$,
$$
fg = 0 \quad\implies\quad \theta(f)g = 0.
$$
Consequently, local property is weaker than $C_0(\Omega)$-linearity.
In the case when the domain and the range bundles are over different base spaces, a more general notion is
defined; namely, \emph{disjointness preserving}, or \emph{separating} (see Section \ref{s:separating}).

\medskip

Local and disjointness preserving linear maps are found in many
researches in analysis. For example, a theorem of Peetre  \cite{Pe60}
states that local linear maps of the space of smooth functions
defined on a manifold modelled on $\mathbb{R}^n$ are exactly linear
differential operators (see, e.g., \cite{Na85}). This is further
extended to the case of vector-valued differentiable functions
defined on a finite dimensional manifold by Kantrowitz and Neumann
\cite{KN01} and Araujo  \cite{Ar04}.

\medskip

In the topological setting,
similar results have been obtained.  Local linear maps of the space
of continuous functions over a locally compact Hausdorff space are
multiplication operators, while disjointness preserving (separating) linear maps
between two such spaces over possibly different base spaces are
weighted composition operators
 (see, e.g., \cite{Ab83, Ar83, Pa84, Ja90, FH94,JW96}).  Among many
interesting questions arising from these two notions, quite a few
efforts has been put on the automatic continuity of such maps.  See,
e.g., \cite{AN80, BNT88, Ja90, JW96} for the scalar case, and
\cite{GJW03, AJ03studia, Ar04, AT05} for the vector-valued case.

\medskip

In this paper, we extend this context to local or separating linear maps between spaces of continuous
sections of vector bundles.
Note that similar to the correspondence developed by Swan \cite{Sw62} between
finite dimensional vector bundles over a locally compact Hausdorff space $\Omega$
and certain $C_0(\Omega)$-modules, the spaces of continuous sections of ``Banach bundles''
are certain Banach $C_0(\Omega)$-modules (see, e.g., \cite{DG83}, and Section \ref{s:prelim} below).

\medskip

One of the original motivation behind this work
is to investigate up to what extent will a local linear map between two
Banach $C_0(\Omega)$-modules be $C_0(\Omega)$-linear.
Surprisingly, on top
of finding that such maps are ``nearly $C_0(\Omega)$-linear'', we
find that they are also ``nearly bounded''. In fact, it is well
known that there are many unbounded $\mathbb{C}$-linear maps from an
infinite dimensional Banach space to another Banach space and so, if
$S$ is a finite set, there are many unbounded $C(S)$-module maps
from certain Banach $C(S)$-module to another Banach $C(S)$-module.
The interesting thing we discovered is that the above is, in many
cases, the ``only obstruction'' to the automatic boundedness of
$C_0(\Omega)$-module maps (see Proposition \ref{prop:fin-unbdd} as
well as Theorems \ref{thm:decomp} and \ref{thm:loc>mod}).

\medskip

More precisely, if $\theta$ is a local $\mathbb{C}$-linear map (not
assumed to be bounded) from an essential Banach $C_0(\Omega)$-module
$E$ to another such module $F$, then $\theta$ is ``nearly
$C_0(\Omega)$-linear'', in the sense that the induced map $\ti
\theta: E\rightarrow \ti F$ is a $C_0(\Omega)$-module map (where
$\ti F$ is the image of $F$ in the space of $C_0$-sections on the
canonical ``(H)-Banach bundle'' associated with $F$; see Section \ref{s:prelim}).
Moreover, $\theta$ is ``nearly bounded'' in the
sense that there exists a finite subset $S \subseteq \Omega$ such
that
$$
\sup_{\omega\in \Omega\setminus S} \sup_{\underset{\|e\| \leq
1}{e\in E;}} \bl\|\ti\theta(e)(\omega)\br\|\ <\ \infty.
$$
Furthermore,
if $F$ is ``$C_0(\Omega)$-normed'' (in particular, if $F$ is a Hilbert
$C_0(\Omega)$-module), then the finite set $S$ consists
of isolated points in $\Omega$, and
$$\theta\ =\ \theta_0 \oplus \bigoplus_{\omega\in S}\theta_\omega$$
where $\theta_0: E_{\Omega\setminus S} \rightarrow
F_{\Omega\setminus S}$ is a bounded $C_0(\Omega\setminus S)$-linear
map (where $E_{\Omega\setminus S}$ and $F_{\Omega\setminus S}$ are
the canonical essential Banach $C_0(\Omega\setminus S)$-modules
induced from $E$ and $F$ respectively) and $\theta_\omega$ are
(unbounded) $\mathbb{C}$-linear maps (see Theorems \ref{thm:loc>mod}
and \ref{thm:decomp}). Consequently, if $\Omega$ contains no isolated
point and $F$ is $C_0(\Omega)$-normed, then $\theta$ is
automatically bounded. As another application, if $X$ and $Y$ are
two Banach spaces and if $\theta_k: X\rightarrow Y$ is a sequence of
$\mathbb{C}$-linear maps (not assumed to be bounded) such that for
any $(x_n)\in c_0(X)$, we have $(\theta_n(x_n))\in c_0(Y)$, then
there exists $n_0$ with
$$\sup_{n\geq n_0} \|\theta_n\|\ <\ \infty.$$

\medskip

On the other hand, we will also study $\mathbb{C}$-linear maps
between two Banach modules over two different base spaces. In this
case, we will consider ``separating'' maps instead of local maps.
More precisely, if $\Omega$ and $\Delta$ are two
locally compact Hausdorff spaces, $E$ is a ``full'' essential Banach
$C_0(\Omega)$-module (see Remark \ref{rem:FV}(b)), and $F$ is a ``full'' Banach
$C_0(\Delta)$-normed module, then for any bijective linear map
$\theta: E\rightarrow F$ (not assumed to be bounded) with both
$\theta$ and $\theta^{-1}$ being separating, there exists a
homeomorphism $\sigma:\Delta \rightarrow \Omega$ such that
$\theta(e\cdot \varphi) = \theta(e)\cdot \varphi\circ \sigma$ ($e\in
E, \varphi\in C_0(\Omega)$), and there exists a finite set $S$
consisting of isolated points of $\Delta$ such that the restriction
of $\theta$ from $E_{\Omega\setminus \sigma(S)}$ to
$F_{\Delta\setminus S}$ is bounded.

\medskip

This paper is organised as follows.
In Section \ref{s:prelim}, we will first
collect some basic facts about the correspondence between Banach bundles and Banach $C_0(\Omega)$-modules.
In Section \ref{s:lemmas}, we will show
two technical lemmas concerning ``near $C_0(\Omega)$-linearity'' and
``near boundedness'' of certain mappings.  Section \ref{s:local} is devoted to
automatic $C_0(\Omega)$-linearity and automatic boundedness of local linear
mappings, while Section \ref{s:separating} is devoted to the automatic boundedness of
bijective biseparating linear mappings between Banach modules over
different base spaces.  Finally, as an attempt to a further generalisation,
we show in the Appendix that for an arbitrary C*-algebra $A$, every \emph{bounded}
local linear map from a Banach $A$-module into a Hilbert $A$-module is $A$-linear.
The boundedness assumption can be removed in the case when $A$ is finite dimension (Corollary \ref{cor:local>mod}).

\medskip

\section{Preliminaries and Notations}\label{s:prelim}

\medskip

Let us first recall
(mainly from \cite{DG83}) some basic terminologies and results
concerning Banach modules and Banach bundles.

\medskip

\begin{ntn}
\label{ntn-init}
In this article, $\Omega$ and $\Delta$ are two locally compact
Hausdorff spaces,  $E$ is an essential Banach $C_0(\Omega)$-module,
$F$ is an essential Banach $C_0(\Delta)$-module, and $\theta:
E\rightarrow F$ is a $\mathbb{C}$-linear map (not assumed to be
bounded).
Furthermore,
$\Omega_\infty$ and $\Delta_\infty$ are the one-point
compactifications of $\Omega$ and $\Delta$ respectively. We denote by
$\CN_\Omega(\omega)$ the set of all compact neighbourhoods of an
element $\omega$ in $\Omega$, and by ${\rm Int}_\Omega(S)$ the set
of all interior points of a subset $S$ in $\Omega$. Moreover, if
$U,V\subseteq \Omega$ such that the closure of $V$ is a compact
subset of ${\rm Int}_\Omega(U)$, we denote by $\CU_\Omega(V,U)$ the
collection of all $\lambda\in C_c(\Omega)$ with $0\leq \lambda \leq
1$, $\lambda \equiv 1$ on $V$ and the support of $\lambda$ lies inside ${\rm Int}_\Omega(U)$.
\end{ntn}

\medskip

\begin{defn}
Let $\Xi$ be a Hausdorff space and $p: \Xi\rightarrow \Omega$ be a
surjective continuous open map. Suppose that for each $\omega\in
\Omega$,
\begin{enumerate}[1).]
\item there exists a complex Banach space structure on $\Xi_\omega := p^{-1}(\omega)$ such that
its norm topology coincides with the topology on $\Xi_\omega$ (as a topological subspace of $\Xi$);

\item $\{W(\epsilon, U): \epsilon > 0, U\in \CN_\Omega (\omega)\}$
forms a neighbourhood basis for the zero element $0_\omega\in
\Xi_\omega$ where $W(\epsilon, U) := \{\xi\in p^{-1}(U): \|\xi\| <
\epsilon \}$;

\item the maps $\mathbb{C} \times \Xi\rightarrow \Xi$ and
$\{(\xi, \eta)\in \Xi\times \Xi: p(\xi) = p(\eta)\} \rightarrow \Xi$
given respectively, by the scalar multiplications and the additions
are continuous.
\end{enumerate}
Then $(\Xi, \Omega, p)$ (or simply, $\Xi$) is called an
\emph{(H)-Banach bundle} (respectively, an \emph{(F)-Banach bundle})
over $\Omega$ if $\xi\mapsto \|\xi\|$ is an upper-semicontinuous
(respectively, continuous) map from $\Xi$ to $\mathbb{R}_+$.
In this case,
$\Omega$ is called the \emph{base space} of $\Xi$, the map $p$ is
called the \emph{canonical projection} and $\Xi_\omega$ is called
the \emph{fibre} over $\omega\in \Omega$.
\end{defn}

\medskip

If $\Xi$ is an (H)-Banach bundle over $\Omega$ and $\Omega_0\subseteq
\Omega$ is an open set, then
$$\Xi_{\Omega_0}\ :=\ \ p^{-1}(\Omega_0)$$
is an (H)-Banach bundle over $\Omega_0$ and is called the
\emph{restriction of $\Xi$ to $\Omega_0$}.
If $\Xi$ is an (F)-Banach bundle, then so is $\Xi_{\Omega_0}$.

\medskip

\begin{defn}
If $\Lambda$ is an (H)-Banach bundle over $\Delta$, a map
$\rho: \Xi \rightarrow \Lambda$ is called a \emph{fibrewise linear map covering a map $\sigma: \Omega\rightarrow \Delta$}
if $\rho(\Xi_\omega) \subseteq \Lambda_{\sigma(\omega)}$ and the restriction $\rho_\omega: \Xi_\omega \rightarrow \Lambda_{\sigma(\omega)}$ is linear.
Moreover, a fibrewise linear map $\rho$ covering a continuous map $\sigma: \Omega\rightarrow \Delta$ is called a \emph{Banach bundle map} if $\rho$ is continuous.
A Banach bundle map $\rho$ is said to be \emph{bounded} if
$\sup_{\underset{\|\xi\|\leq 1}{\xi\in \Xi;}} \|\rho(\xi)\| <
\infty$.
\end{defn}

\medskip

For any map $e: \Omega \rightarrow \Xi$, we denote
$$|e|(\omega)\ :=\ \|e(\omega)\| \qquad (\omega\in \Omega).$$
Such an $e$ is called a \emph{$C_0$-section} on $\Xi$ if $e$ is
continuous, $p(e(\omega)) = \omega$ ($\omega\in \Omega$), and for
any $\epsilon > 0$, there exists a compact set $C\subseteq \Omega$
such that $|e|(\omega) < \epsilon$ ($\omega\in \Omega\setminus C$).
We put
$$
\Gamma_0(\Xi)\ :=\ \{e: \Omega \rightarrow \Xi \mid e {\rm\ is\ a\ }
C_0{\rm -section\ on\ } \Xi\}.
$$ Note that $|e|$ is always upper
semi-continuous for every $e\in \Gamma_{0}(\Xi)$ and $\Xi$ is an
(F)-Banach bundle if and only if all such $|e|$ are continuous.

\medskip

Next, we recall some terminologies and properties concerning an
essential Banach (right) $C_0(\Omega)$-module $E$ (regarded as a unital Banach $C(\Omega_\infty)$-module).
For any $\omega\in
\Omega_\infty$ and $S\subseteq \Omega_\infty$, we denote
$$
K_S\ :=\ \{\varphi\in C(\Omega_\infty): \varphi(S) = \{0\}\}, \quad
K_S^E\ :=\ \overline{E \cdot K_S} \quad {\rm and} \quad I^E_\omega\
:=\ \bigcup_{V\in \CN_{\Omega_\infty}(\omega)}K^E_V.
$$
For simplicity, we set $K_\omega^E := K_{\{\omega\}}^E$.
Note that $K_\infty^E = E$ because $E$ is an essential Banach $C_0(\Omega)$-module.
By
\cite[p.37]{DG83}, there exists an (H)-Banach bundle $\check \Xi^E$
over $\Omega_\infty$ with $\check \Xi^E_\omega = E / K_\omega^E$.
Since $\check \Xi^E_\infty = \{0\}$, if we set $\Xi^E := p^{-1}(\Omega)$, then $\Gamma_0(\Xi^E) \cong \Gamma_0(\check \Xi^E)$ under the canonical identification.
Furthermore, there exists a contraction
$$\sim\ :\ E\ \longrightarrow\ \Gamma_{0}(\Xi^E)$$
such that
$\ti e(\omega) = e + K_\omega^E$. We put $\ti E$ to be the
closure of the
image of
$\sim$.

On the other hand, if $\theta$ is as in Notation \ref{ntn-init}, we define
$$\ti \theta: E\rightarrow \ti F\qquad {\rm by} \qquad \ti
\theta(e) = \widetilde{\theta(e)} \quad (e\in E).$$

\medskip

\begin{defn}
Let $E$ be an essential Banach $C_0(\Omega)$-module.

\smnoind (a) $E$ is called a \emph{Banach $C_0(\Omega)$-convex
module} if for any $\varphi,\psi\in C(\Omega_\infty)_+$ with
$\varphi + \psi =1$, one has $\|x\varphi + y\psi\| \leq
\max\{\|x\|,\|y\|\}$.

\smnoind (b) $E$ is called a \emph{Banach $C_0(\Omega)$-normed
module} if there exists a map $|\cdot|: E \rightarrow C_0(\Omega)_+$
such that for any $x,y\in X$ and $a\in A$,
\begin{enumerate}[i).]
\item $|x+y| \leq |x| + |y|$;
\item $|xa| = |x||a|$;
\item $\|x\| = \| |x| \|$.
\end{enumerate}
\end{defn}

\medskip

Recall that every Hilbert $C_0(\Omega)$-module is
a Banach $C_0(\Omega)$-normed module, and every Banach $C_0(\Omega)$-normed module is
$C_0(\Omega)$-convex. On the other hand, an essential Banach
$C_0(\Omega)$-module $E$ is $C_0(\Omega)$-convex if and only if
$\sim$ is an isometric isomorphism onto $\Gamma_{0}(\Xi^E)$ (see
e.g. \cite[Theorem 2.5]{DG83}).
In this case, we will not distinguish $E$ and $\Gamma_{0}(\Xi^E)$.
Furthermore, $E$ is
$C_0(\Omega)$-normed if and only if $E$ is $C_0(\Omega)$-convex and
$\Xi^E$ is an (F)-Banach bundle (see e.g. \cite[~p.48]{DG83}).

\medskip

For any open subset $\Omega_0\subseteq \Omega$, we set $E_{\Omega_0}
:= K^E_{\Omega\setminus \Omega_0}$ and $\ti E_{\Omega_0} := \Gamma_{0}(\Xi^E_{\Omega_0})$. One can regard
$K_{\Omega\setminus {\Omega_0}}^E$ as an essential Banach
$C_0(\Omega_0)$-module under the identification $C_0(\Omega_0)\cong
K_{\Omega\setminus \Omega_0}$. Note that if $E$ is
$C_0(\Omega)$-convex, then $\ti E_{\Omega_0} = E_{\Omega_0}$.

\medskip

\begin{rem}
\label{count>inj} (a) Let $E$ be a Banach $C_0(\Omega)$-convex
module and $0_\omega$ is the zero element in the fibre
$\Xi^E_\omega$ ($\omega\in \Omega$). It is well-known that
$\omega\mapsto 0_\omega$ is a continuous map from $\Omega$ into
$\Xi^E$. Thus, if $\{\omega_i\}_{i\in I}$ is a net in $\Omega$
converging to $\omega_0\in \Omega$ and $e\in \bigcap_{i\in I}
K_{\omega_i}^E$, then $e\in K_{\omega_0}^E$. Consequently, if
$e\notin K_\omega^E$, there exists $U\in \CN_\Omega(\omega)$ such
that $e\notin K_\alpha^E$ for any $\alpha\in U$.

\noindent (b) For any $\omega\in \Omega$ and $e\in K_\omega^E$, there exists a net $\{e_V\}_{V\in \CN_{\Omega}(\omega)}$ such that $e_V\in K_V^E$ and $\|e - e_V\| \rightarrow 0$.

\smnoind
(c) Let $\Omega = \{\omega_1,\omega_2,...\}$ be a
countable compact Hausdorff space and $E$ be a Banach
$C(\Omega)$-module.
Then $$\bigcap_{\omega\in \Omega} K_\omega^E\ =\ \{0\},$$
or equivalently, the map $\sim$ is injective. In fact,
consider any $e\in \bigcap_{\omega\in \Omega} K_\omega^E$ and any
$\epsilon > 0$. For $k\in \mathbb{N}$, there exists
$\bar\varphi_k\in K_{\{\omega_k\}}$ with $\| e - e\bar\varphi_k\| <
\epsilon/2^{k+1}$. Thus, there exists $\varphi_k\in C(\Omega)$ with
$\varphi_k$ vanishing on an open neighbourhood $V_k$ of $\omega_k$
and $\|e - e \varphi_k\| < \epsilon/ 2^{k}$. Now, consider a finite
subcover $\{V_1,...,V_n\}$ for $\Omega$ and a continuous  partition of unity
$\{\psi_1,...\psi_n\}$ subordinated to $\{V_1,...,V_n\}$. Then
$\|e\|
= \bl\|e - e \sum_{k=1}^n \varphi_k \psi_k\br\|
\leq \sum_{k=1}^n \|e - e\varphi_k\|
< \epsilon$.
\end{rem}

\bigskip

\section{Some technical results}\label{s:lemmas}

\medskip

In this section, we will give two technical lemmas
(\ref{lem:C_0-lin} and \ref{lem:auto-bdd}) which are the crucial ingredients for all the results in this paper.
Before presenting them, let us give another automatic continuity type lemma that is needed for those two essential lemmas.

\medskip

\begin{lem}
\label{lem:cont-sigma} $\FZ_\theta := \{\nu\in \Delta: \ti
\theta(e)(\nu) = 0 {\rm\ for\ all\ } e\in E\}$ is a closed subset (where $\ti\theta$ is as in Section \ref{s:prelim}).
Moreover,
if $\sigma: \Delta_\theta\rightarrow \Omega_\infty$ (where $\Delta_\theta := \Delta \setminus \FZ_\theta$) is a map satisfying
$\theta(I_{\sigma(\nu)}^E)\subseteq K_\nu^F$ ($\nu\in
\Delta_\theta$), then $\sigma$ is continuous.
\end{lem}
\begin{prf}
It follows from Remark \ref{count>inj}(a) that $\FZ_\theta$ is
closed. Suppose on the contrary, that there exists a net
$\{\nu_i\}_{i\in I}$ in $\Delta_\theta$ that converges to $\nu_0\in
\Delta_\theta$ but $\sigma(\nu_i) \nrightarrow \sigma(\nu_0)$. Then
there are $U,W\in \CN_{\Omega_\infty}(\sigma(\nu_0))$ with $U\subseteq {\rm Int}_{\Omega_\infty}(W)$ and $\{i\in I:
\sigma(\nu_i)\notin {\rm Int}_\Omega(W)\}$ being cofinal.
As
$\Omega_\infty$ is compact, by passing to a subnet if necessary, we
can assume that $\{\sigma(\nu_i)\}$ converges to an element
$\omega\in \Omega_\infty$, and there exists $V\in
\CN_{\Omega_\infty}(\omega)$ with $V \cap U = \emptyset$. Pick any
$e\in E$ and $\varphi\in \CU_{\Omega_\infty}(V,\Omega_\infty
\setminus U)$. Since $\sigma(\nu_i) \rightarrow \omega$, we see that
$e(1-\varphi)\in I_{\sigma(\nu_i)}^E$ when $i$ is large enough and so eventually,
$$
\ti \theta(e(1-\varphi))(\nu_i)\ =\ 0
$$
(by the hypothesis). By Remark \ref{count>inj}(a), we see that $\ti
\theta(e(1-\varphi))(\nu_0) = 0$. On the other hand, we have
$\theta(e\varphi) \in K_{\nu_0}^F$ (because $e\varphi\in
I_{\sigma(\nu_0)}^E$) and $\theta(e)\in K_{\nu_0}^F$, which gives
the contradiction that $\nu_0\in \FZ_\theta$.
\end{prf}

\medskip

\begin{rem}
\label{rem:FV} (a) Note that for any $\nu\in \FZ_\theta$, one has
\begin{equation}
\label{char-FV} \theta(I_{\omega}^E)\subseteq K_\nu^F \qquad
(\omega\in \Omega).
\end{equation}
Consequently, if we extend $\sigma$ in Lemma \ref{lem:cont-sigma} by setting  $\sigma(\nu)$ arbitrarily for each $\nu\in \FZ_\theta$, then $\theta(I_{\sigma(\nu)}^E)\subseteq K_\nu^F$ ($\nu\in \Delta$) but one should not expect such $\sigma$ to be continuous.

\smnoind (b) $\theta$ is said to be \emph{full} if $\FZ_\theta =
\emptyset$.
Moreover, $E$ is said to be \emph{full} if $\id:
E\rightarrow E$ is full (or equivalently, $E \neq K_\omega^E$ for any $\omega\in
\Omega$).

\smnoind (c) One can use our proof for Lemma \ref{lem:cont-sigma}
to give the following (probably known) result:
\begin{quotation}
Suppose that $\sigma: \Delta \rightarrow \Omega$ is a map and $\Phi:
C_0(\Omega)\rightarrow C_b(\Delta)$ is a $\mathbb{C}$-linear map
satisfying $\Phi(\lambda\cdot\psi) = \Phi(\lambda)\cdot (\psi\circ
\sigma)$ ($\lambda,\psi\in C_0(\Omega)$), and for any $\nu\in
\Delta$, there exists $\lambda\in C_0(\Omega)$ with
$\Phi(\lambda)(\nu)\neq 0$. Then $\sigma$ is continuous.
\end{quotation}
\end{rem}

\medskip

\begin{lem}
\label{lem:C_0-lin} Let $\sigma: \Delta_\theta\rightarrow \Omega$ be
a map satisfying $\theta(I_{\sigma(\nu)}^E)\subseteq K^F_\nu$
($\nu\in \Delta_\theta$).

\smnoind (a) If $\FU_{\theta} := \bl\{\nu\in \Delta: \sup_{\|e\|
\leq 1} \|\ti \theta(e)(\nu)\| = \infty\br\}$, then $\FU_\theta
\subseteq \Delta_\theta$,
$$\sup_{\nu\in \Delta\setminus \FU_{\theta};\|e\| \leq 1}
\|\ti \theta(e)(\nu)\|\ <\ \infty$$ (we use the convention that
$\sup \emptyset = 0$) and $\sigma(\FU_{\theta})$ is a finite set.

\smnoind (b) If $\FN_{\theta,\sigma} := \left\{\nu\in \Delta_\theta:
\theta(K_{\sigma(\nu)}^E)\nsubseteq K_\nu^F\right\}$, then
$\FN_{\theta,\sigma}\subseteq \FU_{\theta}$ and
$\sigma(\FN_{\theta,\sigma})$ consists of non-isolated points in
$\Omega$.

\smnoind (c) If, in addition, $\sigma$ is an injection sending
isolated points in $\Delta_\theta$ to isolated points in $\Omega$,
then $\ti \theta(e\cdot \varphi) = \ti \theta(e) \cdot
\varphi\circ\sigma$ ($e\in E, \varphi\in C_0(\Omega)$).
\end{lem}
\begin{prf}
(a) The first conclusion is clear.
We put $Y$ to be the $c_0$-direct sum $\bigoplus_{\nu\in \Delta}^{c_0}
\Xi^F_\nu$.
For every $\nu\in \Delta\setminus \FU_\theta$, one can regard $e\mapsto \ti\theta(e)(\nu)$ as a bounded $\mathbb{C}$-linear map from $E$ into $Y$ (note that $\left\|\ti\theta(e)(\nu)\right\| \leq \|\theta(e)\|$), the uniform boundedness principle will
give the second conclusion.
Assume now that $\sigma(\FU_\theta)$ is
infinite. For $n=1$, we can find $\nu_1\in \Delta$ as well as
$e_1\in E$ with $\|e_1\| \leq 1$ and $\left\|\ti \theta(e_1)(\nu_1)\right\| >
1$. Inductively, we can find $\nu_n\in \Delta$ and $e_n\in E$ such
that $$\sigma(\nu_n)\ \neq\ \sigma(\nu_k)\ (k=1,...,n-1),\quad
\|e_n\| \leq 1 \quad {\rm and}\quad \|\ti \theta(e_n)(\nu_n)\| >
n^3.$$
There exist $n_1\in \mathbb{N}$ and $U_1\in
\CN_\Omega(\sigma(\nu_{n_1}))$ such that $\{n\in\mathbb{N}: n > n_1 {\rm \ and\ } \sigma(\nu_n)\notin U_1\}$ is infinite. Inductively, we can find a
subsequence $\{\nu_{n_k}\}$ and $U_k\in
\CN_\Omega(\sigma(\nu_{n_k}))$ ($k\in \mathbb{N}$) such that
$U_k\cap U_l = \emptyset$ for distinct $k,l \in \mathbb{N}$. Without
loss of generality, we assume that $n_k = k$. Pick $V_n\in
\CN_\Omega(\sigma(\nu_{n}))$ such that $V_n$ is subset of ${\rm
Int}_{\Omega}(U_n)$. Consider $\lambda_n\in \CU_\Omega(V_n, U_n)$
($n\in \mathbb{N}$) and notice that $\|e_n\lambda_n^2\| \leq 1$.
Define $e := \sum_{k=1}^\infty \frac{e_k\lambda_k^2}{k^2} \ \in\ E$ and take $n\in \mathbb{N}$.
Since
$$n^2e - e_n\lambda_n^2\ =\ n^2\left(\sum_{k\neq n} \frac{e_k \lambda_k}{k^2}\right)\left(\sum_{k\neq n} \lambda_k\right)\ \in\ K^E_{U_n},$$
we have $n^2\ti\theta(e)(\nu_n) = \ti \theta(e_n\lambda_n^2)(\nu_n)$
(by the hypothesis). On the other hand, as $e_n - e_n\lambda^2_n =
e_n(1 - \lambda^2_n)\in K^E_{V_n}$, we have,
$$
\left\|\ti \theta (e)\right\|\ \geq\ \left\|\ti
\theta(e)(\nu_n)\right\|\ = \ \frac{1}{n^2}\left\|\ti
\theta(e_n\lambda_n^2)(\nu_n)\right\| \ = \ \frac{1}{n^2}\left\|\ti
\theta(e_n)(\nu_n)\right\| \ >\ n,
$$
which contradicts the finiteness of $\|\ti \theta (e)\|$.

\smnoind (b) Consider $\nu\in \Delta \setminus \FU_{\theta}$ and denote
$\kappa := \sup_{\|e\| \leq 1} \left\|\ti \theta(e)(\nu)\right\| <
\infty$.
Pick any $e\in K_{\sigma(\nu)}^E$ and $e_V\in K_V^E$
($V\in \CN_\Omega(\sigma(\nu))$) with $\|e_V - e\| \rightarrow 0$ (Remark \ref{count>inj}(b)).
As $\theta(e_V)\in K_{\nu}^F$, one has
$$
\bl\|\ti \theta(e)(\nu)\br\|\ =\ \bl\|\ti\theta(e-e_V)(\nu)\br\|\
\leq\ \kappa \bl\|e-e_V\br\|,
$$
which shows that $\nu\in \Delta \setminus \FN_{\theta,\sigma}$.
Now, if $\sigma(\nu)$ is an isolated point in $\Omega$, then $\{\sigma(\nu)\} \in \CN_\Omega(\sigma(\nu))$, and we have the contradiction that  $\theta(K_{\sigma(\nu)}^E)\subseteq K_\nu^F$.
This gives second statement.

\smnoind (c) For any $\nu\in \Delta \setminus \FN_{\theta,\sigma}$
and $e\in E$, we have $e\varphi - e\varphi(\sigma(\nu)) = e (\varphi
- \varphi(\sigma(\nu))1) \in K_{\sigma(\nu)}^E$. Thus,
\begin{equation}
\label{mod-map} \ti \theta(e\varphi)(\nu)\ =\ \ti\theta(e)(\nu)
\varphi(\sigma(\nu)) \qquad (e\in E, \nu\in \Delta \setminus
\FN_{\theta,\sigma}).
\end{equation}
In particular, \eqref{mod-map} is true when $\nu\in \Delta\setminus
\FU_{\theta}$ (by part (b)) or when $\nu\in \FU_{\theta}$ is an
isolated point of $\Delta_\theta$ (by the hypothesis as well as part
(b)).
Suppose that $\nu\in \FU_{\theta}$ is a
non-isolated point of $\Delta_\theta$.
As $\sigma$ is injective, part (a) implies that $\FU_{\theta}$ is a
finite set.
Hence, there exists a net $\{\nu_i\}$ in
$\Delta_\theta\setminus \FU_{\theta}$ converging to $\nu$. Now, by
Lemma \ref{lem:cont-sigma},
$$
\ti \theta(e\varphi)(\nu)\ =\ \lim \ti \theta(e\varphi)(\nu_i)\ =\
\lim \ti \theta(e)(\nu_i)\varphi(\sigma(\nu_i))\ =\ \ti
\theta(e)(\nu)\varphi(\sigma(\nu)).
$$
\end{prf}

\medskip

\begin{rem}
\label{rem:remove-FV} Note that since $\FZ_\theta$ is closed, isolated points in $\Delta_\theta$ are the same as isolated points of $\Delta$.
Moreover, for any $\nu\in \FZ_\theta$, we have $\sup_{\|e\|\leq 1}
\|\ti\theta(e)(\nu)\| = 0$, and \eqref{char-FV} holds.
Therefore, Lemma \ref{lem:C_0-lin} remains valid if we replace all the $\Delta_\theta$ with $\Delta$ (in fact, the current form is stronger as any injection on $\Delta$ restricted to an injection on $\Delta_\theta$).
The same is true for all the remaining results
in this section.
\end{rem}

\medskip

If $\sigma$ is injective, then $\FU_{\theta}$ is finite and we have our first nearly automatically boundedness result which states that if $\theta$ is a ``module map through an injection $\sigma: \Delta\rightarrow \Omega$'' (one can relax this slightly to an injection on $\Delta_\theta$), then $\theta$ is ``bounded after taking away finite number of points from $\Delta$''.

\medskip

\begin{prop}
\label{prop:fin-unbdd} Let $\Omega$ and $\Delta$ be two locally
compact Hausdorff spaces. Let $E$ and $F$ be an essential Banach
$C_0(\Omega)$-module and an essential Banach $C_0(\Delta)$-module
respectively, and let $\theta: E\rightarrow F$ be a
$\mathbb{C}$-linear map (not assumed to be bounded). Suppose that
$\sigma: \Delta_\theta \rightarrow \Omega$ is an injection satisfying $\theta(e\cdot \varphi)(\nu) =
\theta(e)(\nu)\varphi(\sigma(\nu))$ ($e\in E, \varphi\in
C_0(\Omega), \nu \in \Delta_\theta$). Then
there exist a finite
subset $T\subseteq \Delta$ and
$\kappa > 0$ such that
$$\sup_{\nu\in \Delta \setminus T} \|\ti \theta(e)(\nu)\|\ \leq \
\kappa \|e\| \qquad (e\in E).$$
\end{prop}

\medskip

\begin{lem}
\label{lem:auto-bdd} Let $\sigma:\Delta_\theta\rightarrow \Omega$ be a map
satisfying $\theta(I_{\sigma(\nu)}^E)\subseteq K_\nu^F$ ($\nu\in
\Delta_\theta$). Suppose, in addition, that $F$ is a Banach
$C_0(\Delta)$-normed module.

\smnoind (a) $\FN_{\theta,\sigma}$ is an open subset of $\Delta$.

\smnoind (b) If $\sigma$ is injective, then $\FU_{\theta}$ is a
finite set consisting of isolated points of $\Delta$.
If, in addition, $\FU_\theta \neq \Delta$, then $F =
F_{\Delta\setminus \FU_{\theta}} \oplus \bigoplus_{\nu\in
\FU_{\theta}}\Xi^F_\nu$ and
$$
\theta_0 := P_{\theta,\sigma} \circ \theta|_{E_{\Omega\setminus
\sigma(\FU_{\theta})}}:E_{\Omega\setminus \sigma(\FU_{\theta})}
\rightarrow F_{\Delta\setminus \FU_{\theta}}
$$
is a bounded linear map (where $P_{\theta,\sigma}: F \rightarrow
F_{\Delta\setminus \FU_{\theta}}$ is the canonical projection) such
that
\begin{equation}
\label{re-mod} \theta_0(e\cdot \varphi) = \theta_0(e)\cdot
\varphi\circ\sigma \qquad (e\in E_{\Omega\setminus
\sigma(\FU_{\theta})}, \varphi\in C_0(\Omega\setminus
\sigma(\FU_{\theta})))
\end{equation}
(note that the value of $\sigma$ on $\FZ_\theta$ can be set arbitrarily).
\end{lem}
\begin{prf}
Notice, first of all, that as $F$ is $C_0(\Omega)$-convex, one can regard
$\ti \theta = \theta$.

\smnoind (a) As $\Delta_\theta$ is open in $\Delta$ and
$\FN_{\theta,\sigma}\subseteq \FU_{\theta} \subseteq \Delta_\theta$, it suffices to show that
$\FN_{\theta,\sigma}$ is open in $\Delta_\theta$.
By Lemma \ref{lem:C_0-lin}(a),
$$
\kappa\ :=\ \sup_{\nu\notin \FU_\theta} \sup_{\|e\| \leq 1}
\|\theta(e)(\nu)\|\ < \ \infty.
$$
Let $\{\nu_i\}_{i\in I}$ be a net in $\Delta_\theta\setminus
\FN_{\theta,\sigma}$ converging to $\nu_0\in \Delta_\theta$, and $e$
be an arbitrary element in $K_{\sigma(\nu_0)}^E$. By Lemma
\ref{lem:cont-sigma}, we know that $\sigma(\nu_i) \rightarrow
\sigma(\nu_0)$.
Now, we consider two cases separately.
The first case is when $\{\sigma(\nu_i)\}_{i\in I}$ is finite.
In this case, by passing to subnet, we can assume that $\sigma(\nu_i)
= \sigma(\nu_0)$ ($i\in I$).
As $e(\sigma(\nu_0)) = 0$ and $\nu_i\notin \FN_{\theta,\sigma}$, we have $\theta(e)(\nu_i) = 0$ which gives $\theta(e)(\nu_0) = 0$, and so, $\theta(e)\in K_{\nu_0}^F$.
The second case (of $\{\sigma(\nu_i)\}_{i\in I}$ being infinite) can be subdivided into two cases.
More precisely, if there
exists $i_0\in I$ such that $\nu_j \in \FU_\theta$ for every $j\geq
i_0$, then we can assume that $\{\sigma(\nu_i)\}_{i\in I}\subseteq
\sigma(\FU_\theta)$ which is a finite set, and the above
implies that $\theta(e)\in K_{\nu_0}^F$. Otherwise, $\{i\in I:
\nu_i\notin \FU_\theta\}$ is cofinal, and by passing to a subnet, we
may assume that $\nu_i\notin \FU_\theta$ ($i\in I$). For any
$\epsilon > 0$, pick $V\in \CN_\Omega(\sigma(\nu_0))$ and $e_V\in
K_V^E$ with $\|e_V - e\| < \epsilon$. When $i$ is large enough,
$\sigma(\nu_i)\in V$ and $e_V(\sigma(\nu_i)) = 0$. Thus,
$$\|\theta(e)(\nu_i)\|\ =\ \|\theta(e - e_V)(\nu_i)\|\ \leq\ \kappa \epsilon.$$
By the continuity of the norm function on $\Xi^F$, we have
$\|\theta(e)(\nu_0)\| \leq \kappa\epsilon$ which implies that
$\theta(e)(\nu_0) = 0$.

\smnoind (b) By the hypothesis and Lemma \ref{lem:C_0-lin}(a),
one knows that $\FU_{\theta}$ is finite.
Without loss of generality, we assume $\Delta \neq \FU_{\theta}$.
Let
\begin{equation}
\label{sup-Delta_0} \kappa\ :=\ \sup_{\nu\in \Delta\setminus
\FU_{\theta}} \sup_{\|e\|\leq 1} \|\theta(e)(\nu)\|\ <\ \infty.
\end{equation}
Suppose on the contrary that there is $\nu_0\in \FU_{\theta}$ which
is not an isolated point in $\Delta$. As $\FU_{\theta}$ is finite,
there is a net $\{\nu_i\}$ in $\Delta \setminus \FU_{\theta}$ such
that $\nu_i\rightarrow \nu_0$.
By the definition of $\FU_{\theta}$,
there is $e\in E$ with $\|e\| \leq 1$ and $\|\theta(e)(\nu_0)\| >
\kappa + 1$.
However, this will contradict the continuity of
$|\theta(e)|$ (because of \eqref{sup-Delta_0}).
Now, as
$\FU_{\theta}$ is a finite set consisting of isolated points in
$\Delta$ and $F$ is the space of $C_0$-sections on $\Xi^F$, we see
that $$F\ =\ K_{\FU_{\theta}}^F \oplus \bigoplus_{\nu\in
\FU_{\theta}}\Xi^F_\nu.$$
By Lemma \ref{lem:C_0-lin}(b) and the argument of Lemma
\ref{lem:C_0-lin}(c) (more precisely, \eqref{mod-map}), one easily check that
$\theta_0$ will satisfy \eqref{re-mod}.
On the other hand, the
boundedness $\theta_0$ follows from \eqref{sup-Delta_0}.
\end{prf}

\medskip

Observe that in Lemmas \ref{lem:C_0-lin}(c) and
\ref{lem:auto-bdd}(b), one can replace the injectivity of $\sigma$
with the condition that $\sigma^{-1}(\omega)$ is at most finite for
any $\omega\in \Omega$.

\medskip

The following is our second nearly automatically boundedness result
that applies, in particular, when $F$ is a Hilbert
$C_0(\Delta)$-module.

\medskip

\begin{thm}
\label{thm:decomp} Let $\Omega$ and $\Delta$ be two locally compact
Hausdorff spaces. Let $E$ be an essential Banach
$C_0(\Omega)$-module, let $F$ be an essential Banach
$C_0(\Delta)$-normed module, and let $\theta: E\rightarrow F$ be a
$\mathbb{C}$-linear map (not assumed to be bounded). Suppose that
$\sigma: \Delta_\theta\rightarrow \Omega$ is an injection satisfying $\theta(I_{\sigma(\nu)}^E)\subseteq K^F_\nu$ ($\nu\in \Delta$).

\smnoind (a) If $\Delta$ contains no isolated point, then $\theta$
is bounded.

\smnoind (b) If $\sigma$ sends isolated points in $\Delta_\theta$ to
isolated points in $\Omega$, then $\FN_{\theta,\sigma} = \emptyset$
and
there exist a finite set $T$ consisting of isolated points of
$\Delta$, a bounded linear map $\theta_0:E_{\Omega\setminus
\sigma(T)} \rightarrow F_{\Delta\setminus T}$ as well as linear maps
$\theta_\nu: \Xi^E_{\sigma(\nu)} \rightarrow \Xi^F_\nu$ for all $\nu\in T$ such that $E = E_{\Omega\setminus \sigma(T)} \oplus
\bigoplus_{\nu\in T} \Xi^E_{\sigma(\nu)}$,
$F = F_{\Delta\setminus T} \oplus \bigoplus_{\nu\in T} \Xi^F_\nu$
and $\theta = \theta_0 \oplus \bigoplus_{\nu\in
T}\theta_\nu$.
\end{thm}
\begin{prf}
(a) This follows directly from Lemma \ref{lem:auto-bdd}(b).

\smnoind (b) The first conclusion follows from Lemma
\ref{lem:C_0-lin}(c)
and the second conclusion follows from Lemma \ref{lem:auto-bdd}(b) (notice that we have a sharper conclusion here since $\FN_{\theta,\sigma} = \emptyset$).
\end{prf}

\bigskip

\section{Applications to local linear mappings}\label{s:local}

\medskip

In the section, we will consider the case when $\Delta = \Omega$, $\sigma = \id$, and the $\mathbb C$-linear map $\theta$ is a \emph{local
map} in the sense that $\theta(e) \cdot \varphi = 0$ whenever $e\in
E$ and $\varphi\in C_0(\Omega)$ satisfying $e\cdot \varphi = 0$. It
is obvious that any $C_0(\Omega)$-module map is local.

\medskip

\begin{rem}
\label{rem:loc>good} Suppose that $\theta$ is local. Let
$U,V\subseteq \Omega$ be open sets with the closure of $V$ being a
compact subset of $U$, and consider $\lambda\in \CU_\Omega(V,U)$.
For any $e\in K^E_U$ and any $\epsilon > 0$, there exists
$\varphi\in K_U$ with $\|e - e\varphi\| < \epsilon$. Thus, $e\lambda
= 0$ which implies that $\theta(e)\lambda = 0$ and $\theta(e) =
\theta(e)(1 - \lambda)\in K_V^F$. This shows that $\sigma = \id$
will satisfy the hypothesis in all the results in Section \ref{s:lemmas}.
\end{rem}

\medskip

The following theorem (which follows directly from the results in
Section \ref{s:lemmas} as well as Remark \ref{rem:loc>good}) is our main result
concerning local linear maps.

\medskip

\begin{thm}
\label{thm:loc>mod} Let $\Omega$ be a locally compact Hausdorff
space. Suppose that $E$ and $F$ are essential Banach
$C_0(\Omega)$-modules, and $\theta: E\rightarrow F$ is a local
$\mathbb{C}$-linear map (not assumed to be bounded).

\smnoind (a) $\ti \theta$ is a $C_0(\Omega)$-module map and
there exist a finite
subset $T\subseteq \Delta$ and
$\kappa > 0$ such that
$\sup_{\nu\in \Delta \setminus T} \|\ti \theta(e)(\nu)\| \leq
\kappa \|e\|$ ($e\in E$).

\smnoind (b) If, in addition, $F$ is a Banach $C_0(\Omega)$-normed module, then
$\theta$ is a $C_0(\Omega)$-module map and
there exist a finite set $T$ consisting of isolated points of
$\Omega$, a bounded linear map $\theta_0:E_{\Omega\setminus
T} \rightarrow F_{\Omega\setminus T}$ as well as a linear map
$\theta_\nu: \Xi^E_{\nu} \rightarrow \Xi^F_\nu$ for each $\nu\in T$
such that $E = E_{\Omega\setminus T} \oplus
\bigoplus_{\nu\in T} \Xi^E_{\nu}$,
$F = F_{\Omega\setminus T} \oplus \bigoplus_{\nu\in T} \Xi^F_\nu$
and $\theta = \theta_0 \oplus \bigoplus_{\nu\in
T}\theta_\nu$.
\end{thm}

\medskip

It is natural to ask if one can relax the assumption of $F$ being
$C_0(\Omega)$-normed to $C_0(\Omega)$-convex in the second statement
of Theorem \ref{thm:loc>mod} (in particular, whether it is true that every $C_0(\Omega)$-module map
from an essential Banach $C_0(\Omega)$-module to an essential Banach
$C_0(\Omega)$-convex module is automatically bounded provided that $\Omega$
contains no isolated point). Unfortunately, it is not the case as
can be seen by the following simple example.

\medskip

\begin{eg}
Let $E := C([0,1]) \oplus^\infty X$ and $F := C([0,1]) \oplus^\infty
Y$, where $X$ and $Y$ are two infinite dimensional Banach spaces.
Then $E$ is an essential Banach $C([0,1])$-convex module under the
multiplication: $(e, x)\cdot \varphi = (e\cdot \varphi, x
\varphi(0))$ ($e, \varphi\in C([0,1]); x\in X$). In the same way,
$F$ is an essential Banach $C([0,1])$-convex module. Suppose that
$R: X \rightarrow Y$ is an unbounded linear map and $\theta:
E\rightarrow F$ is given by $\theta(e,x) = (e,R(x))$ ($e\in
C([0,1]); x\in X$). Then $\theta$ is a $C([0,1])$-module map which
is not bounded (as its restriction on $X$ is $R$). In this case, we
have $\FU_{\theta} = \{0\}$.
\end{eg}

\medskip

\begin{cor}\label{cor:bun-al-bdd}
Let $\Omega$ be a locally compact Hausdorff space.
Any local
$\mathbb{C}$-linear $\theta$ from an essential Banach
$C_0(\Omega)$-module to a Hilbert $C_0(\Omega)$-module is a
$C_0(\Omega)$-module map.
Moreover, if $\Omega$ contains no isolated
point, then any such $\theta$ is automatically bounded.
\end{cor}

\medskip

\begin{rem}
Let $L_{C_0(\Omega)}(E;C_0(\Omega))$ (respectively,
$\mathcal{B}_{C_0(\Omega)}(E;C_0(\Omega))$) be the ``algebraic
dual'' (respectively, ``topological dual'') of $E$, i.e. the
collection of all $C_0(\Omega)$-module maps (respectively, bounded $C_0(\Omega)$-module maps)
from $E$ into $C_0(\Omega)$. An application of Corollary
\ref{cor:bun-al-bdd} is that the algebraic dual and the topological
dual of $E$ coincide in many cases:
\begin{quotation}
If $\Omega$ is a locally compact Hausdorff space having no isolated
point and $E$ is an essential Banach $C_0(\Omega)$-module, then
$\mathcal{B}_{C_0(\Omega)}(E;C_0(\Omega)) =
L_{C_0(\Omega)}(E;C_0(\Omega))$.
\end{quotation}
\end{rem}

\medskip

\begin{cor}
Let $\Xi$ and $\Lambda$ be respectively an
(H)-Banach bundle and an (F)-Banach bundle over the same base space
$\Omega$.
If $\rho: \Xi \rightarrow \Lambda$ is a fibrewise linear
map covering $\id$ (without any boundedness nor continuity assumption) such that
$\rho\circ e \in \Gamma_{0}(\Lambda)$ for every $e\in
\Gamma_{0}(\Xi)$, then there exists a finite subset $S\subseteq
\Omega$ consisting of isolated points such that $\rho$ restricts to
a bounded Banach bundle map $\rho_0: \Xi_{\Omega\setminus S}
\rightarrow \Lambda_{\Omega\setminus S}$.
\end{cor}

\medskip

Let $X$ be a Banach space. We denote by $\ell^\infty(X)$ and
$c_0(X)$ the set of all bounded sequences and the set of all
$c_0$-sequences in $X$, respectively.
We recall that $\ell^\infty \cong C(\beta\mathbb{N})$ where $\beta\mathbb{N}$ is the Stone-Cech compactification of $\mathbb{N}$ (which can be identified with the collection of all ultrafilters on $\mathbb{N}$).

\medskip

\begin{prop}
\label{prop:seq-maps} Let $X$ and $Y$ be Banach spaces, and let
$\theta_k: X\rightarrow Y$ ($k\in \mathbb{N}\cup \{\infty\}$) be
linear maps (not assumed to be bounded). For any sequence
$\{x_k\}_{k\in \mathbb{N}}$ in $X$, we put $\theta(\{x_k\}_{k\in
\mathbb{N}}) := \{\theta_k(x_k)\}_{k\in \mathbb{N}}$.

\smnoind (a) If $\theta(c_0(X))\subseteq c_0(Y)$, then there exists
$n_0 \in \mathbb{N}$ such that $\sup_{n \geq n_0} \|\theta_n\| <
\infty$.

\smnoind (b) If $\lim_{k\rightarrow \infty} \theta_k(x_k) =
\theta_\infty(x)$ for any $\{x_k\}_{k\in
\mathbb{N}}\in\ell^\infty(X)$ with $\lim_{k\rightarrow \infty} x_k =
x$, then $\theta_\infty$ is bounded, and there is $n_0 \in
\mathbb{N}$ such that $\sup_{n \geq n_0} \|\theta_n\| < \infty$.

\smnoind (c) Suppose that $\theta(\ell^\infty(X))\subseteq
\ell^\infty(Y)$, and
$\lim_\CF \theta_k(x_k) =0$
for every
$\{x_k\}_{k\in \mathbb{N}}\in \ell^\infty(X)$ and every ultrafilter
$\CF$ on $\mathbb{N}$ with $\lim_\CF x_k =0$. Then there exist
$\CF_1,...,\CF_n \in \beta\mathbb{N}$ with $\sup_{\CF \neq
\CF_1,...,\CF_n} \|\theta_\CF\| < \infty$ (where $\theta_\CF:
\Xi^{\ell^\infty(X)}_\CF \rightarrow \Xi^{\ell^\infty(Y)}_\CF$ is
the induced map). In particular, $\sup_{n \geq n_0} \|\theta_n\| <
\infty$ for some $n_0 \in \mathbb{N}$.
\end{prop}
\begin{prf}
(a) Let $E = c_0(X)$ and $F = c_0(Y)$. Then $\theta$ is a
$C_0(\mathbb{N})$-module map and we can apply Theorem
\ref{thm:loc>mod}.

\smnoind (b) Let $E = C(\mathbb{N}_\infty, X)$ and $F =
C(\mathbb{N}_\infty, Y)$. Then $\theta\oplus \theta_\infty$ is a
well defined $C(\mathbb{N}_\infty)$-module map from $E$ into $F$ and
Theorem \ref{thm:loc>mod} implies this part.

\smnoind (c) Let $E = \ell^\infty(X)$ and $F = \ell^\infty(Y)$. Then
$E$ and $F$ are unital Banach $C(\beta\mathbb{N})$-modules. For any
ultrafilter $\CF\in \beta\mathbb{N}$, one has
$$K^E_\CF\ =\ \{ (x_n)\in E: \lim_\CF x_n = 0 \} \quad {\rm and}\quad K^F_\CF\ =\ \{ (y_n)\in F: \lim_\CF y_n = 0 \}.$$
The first hypothesis shows that $\theta(E) \subseteq F$ and the
second one tells us that $\theta(K^E_\CF) \subseteq K^F_\CF$.
On the other hand, if $n\in \mathbb{N}$ and $\CF_n := \{U\subseteq
\mathbb{N}: n\in U\}$, then
$$K^E_{\CF_n}\ =\ \{(x_k)\in \ell^\infty(X): x_n = 0\}$$
and so, $\theta_{\CF_n} = \theta_n$. Now, this part follows
from Theorem \ref{thm:loc>mod}.
\end{prf}

\medskip

\begin{rem}
Note that if $\CF$ is a free ultrafilter on $\mathbb{N}$, then
$\Xi^{\ell^\infty(X)}_\CF$ and $\Xi^{\ell^\infty(Y)}_\CF$ can be identified with the ultrapowers $X^\CF$ and $Y^\CF$ of $X$ and $Y$ (over $\CF$) respectively. One can
interpret Proposition \ref{prop:seq-maps}(c) as follows:
\begin{quote}
If the sequence $\{\theta_n\}$ as in Proposition \ref{prop:seq-maps} induces canonically a map $\theta:\ell^\infty(X)\rightarrow \ell^\infty(Y)$
as well as a map $\theta_\CF: X^\CF\rightarrow Y^\CF$ for every free ultrafilter
$\CF$ (none of them assumed to be bounded), then for all but a finite number of ultrafilters $\CF$, the map $\theta_\CF$ is bounded and
they have a common bound.
\end{quote}
It can be shown easily that the converse of the above is also true (but we left it to the readers to check the details):
\begin{quotation}
If the sequence $\{\theta_n\}$ is as in Proposition \ref{prop:seq-maps} and there exists $n_0\in \mathbb{N}$ with $\sup_{n\geq n_0} \|\theta_n\| < \infty$, then $\{\theta_n\}$ induces canonically a map from $\ell^\infty(X)$ to
$\ell^\infty(Y)$ as well as a map from $X^\CF$ to $Y^\CF$ for every free ultrafilter $\CF$.
\end{quotation}
\end{rem}

\medskip

Another important point in Theorem \ref{thm:loc>mod} is the
automatic $C_0(\Omega)$-linearity. In fact, it can be shown that for
every $C^*$-algebra $A$, any bounded local linear map from a Banach right
$A$-module into a Hilbert $A$-module is automatically $A$-linear
(see Proposition \ref{prop:local+bounded=modulemaps} in the Appendix).
Theorem \ref{thm:loc>mod} tells us that if $A$ is commutative, then one can relax the assumption of the range space to a Banach $A$-convex module and one can remove the boundedness assumption.
Another application of this theorem is that if $A$ is a finite dimensional
$C^*$-algebras, then every local linear map between any two Banach right
$A$-modules is $A$-linear.

\medskip

\begin{cor}
\label{cor:local>mod} Let $A$ be a finite dimensional $C^*$-algebra.
Suppose that $E$ and $F$ are unital Banach right $A$-modules. If
$\theta: E\rightarrow F$ is a local $\mathbb{C}$-linear map in the
sense of Proposition \ref{prop:local+bounded=modulemaps} (not
assumed to be bounded), then $\theta$ is an $A$-module map.
\end{cor}
\begin{prf}
Pick any $x\in E$ and $a\in A_{sa}$. Let $A_a := C^*(a,1)$. By
Remark \ref{count>inj}(c), both $E$ and $F$ are unital Banach
$A_a$-convex modules. Thus, Theorem \ref{thm:loc>mod} tell us that
$\theta$ is a $A_a$-module map. In particular, $\theta(xa) =
\theta(x)a$.
\end{prf}

\medskip

\begin{rem}
(a) Suppose that $A$ is a unital $C^*$-algebra and $F$ is a unital
Banach right $A$-convex module in the sense $\|x a + y (1-a)\| \leq
\max\{\|x\|, \|y\|\}$ for $x,y\in F$ and $a\in A_+$ with $a\leq
1$.
Then, by the argument of Corollary \ref{cor:local>mod}, all
local linear maps from any unital Banach right $A$-module into $F$ are
automatically $A$-linear.

\smnoind (b) If one can show that for every compact subset
$\Omega\subseteq \mathbb{R}$ and every essential Banach
$C(\Omega)$-module $F$, the map $\sim: F\rightarrow \ti F$ is
injective, then using the argument of Corollary \ref{cor:local>mod},
one can show that for each $C^*$-algebra $A$, all local linear maps between
any two Banach right $A$-modules are $A$-module maps (without
assuming that $\theta$ is bounded).
However, we do not know if it is true.
\end{rem}

\bigskip

\section{Applications to separating mappings}\label{s:separating}

\medskip

In this section, we consider $\Omega$ and $\Delta$ to be possibly
different spaces.
In this case, one cannot define local property any
more, but one has a weaker natural property called separating. More
precisely, $\theta$ is said to be \emph{separating} if
$$
|\ti\theta(e)| \cdot |\ti \theta(g)| = 0, \quad\text{whenever $e,g\in E$ satisfying
$|\ti e| \cdot |\ti g| = 0$}.
$$
In the case when $E = C_0(\Omega)$ and $F=C_0(\Delta)$, this coincides with the well-known notion of disjointness preserving (see e.g. \cite{Ab83, Ar83, Pa84, Ja90, FH94,JW96}).

\medskip

\begin{lem}
\label{lem-diff-base} If $\theta$ is separating, there is a
continuous map $\sigma: \Delta_\theta \rightarrow \Omega_\infty$
such that $\theta(I_{\sigma(\nu)}^E)\subseteq I_\nu^F$ ($\nu\in
\Delta_\theta$).
\end{lem}
\begin{prf}
Set
$$S_\nu\ :=\ \{\omega\in \Omega_\infty: \theta(I_\omega^E) \subseteq I_\nu^F\} \quad (\nu\in \Delta_\theta).$$
Suppose there is $\nu\in \Delta_\theta$ with $S_\nu =
\emptyset$.
Then for each $\omega\in \Omega_\infty$, there exist
$U_\omega\in \CN_{\Omega_\infty}(\omega)$ and $e_\omega\in K^E_{U_\omega}$ with $\theta(e_\omega) \notin
I_\nu^F$. Let $\{U_{\omega_i}\}_{i=1}^n$ be a finite subcover of
$\{U_\omega\}_{\omega\in \Omega_\infty}$ and $\{\varphi_i\}_{i=1}^n$
be a partition of unity subordinate to $\{U_{\omega_i}\}_{i=1}^n$.
Take any $g\in E$.
From $|\widetilde{g\varphi_i}|
|\widetilde{e_{\omega_i}}| = 0$, we obtain $|\ti \theta(g\varphi_i)|
|\ti \theta(e_{\omega_i})| = 0$, which implies that $\ti
\theta(g\varphi_i)(\nu) = 0$ (because of Remark \ref{count>inj}(a)
and the fact that $\theta(e_{\omega_i}) \notin I_\nu^F$).
Consequently,
$$\ti\theta(g)(\nu)\ =\ \sum_{i=1}^n \ti \theta(g\varphi_i)(\nu)\ =\ 0,$$
and we arrive in the contradiction that $\nu\in \FZ_\theta$.
Suppose there is $\nu\in \Delta_\theta$ with $S_\nu$ containing two
distinct points $\omega_1$ and $\omega_2$.
Let $U,V\in \CN_{\Omega_\infty}(\omega_1)$ with $V\subseteq {\rm
Int}_{\Omega_\infty}(U)$ and $\omega_2\notin U$.
For any $\varphi\in \CU_{\Omega_\infty}(V,U)$ and $e\in E$, we have
$e(1-\varphi)\in I_{\omega_1}^E$ and $e \varphi \in I_{\omega_2}^E$
which implies that
$$\theta(e)\ =\ \theta(e(1-\varphi)) + \theta(e\varphi)\ \in\ I_\nu^F.$$
This gives the contradiction that $\nu\in \FZ_\theta$. Therefore, we
can define $\sigma(\nu)$ to be the only point in $S_\nu$, and it is
clear that $\theta\left(I_{\sigma(\nu)}^E\right) \subseteq I_\nu^F$.
Now, the
continuity of $\sigma$ follows from Lemma \ref{lem:cont-sigma}.
\end{prf}

\medskip

\begin{cor}
\label{cor:bundle} Let $\Xi$ be an (H)-Banach bundle over $\Omega$,
let $\Lambda$ be an (F)-Banach bundle over $\Delta$, and let $\rho:
\Xi\rightarrow \Lambda$ be a map (not assumed to be bounded nor
continuous). Suppose that $\sigma: \Delta\rightarrow \Omega$ is an
injection sending isolated points in $\Delta$ to isolated points
in $\Omega$ such that $e \mapsto \rho\circ e\circ \sigma$ defines a
linear map $\theta:\Gamma_{0}(\Xi)\rightarrow \Gamma_{0}(\Lambda)$.
Then there exists a finite set $T$ consisting of isolated points of
$\Delta$ such that the restriction of $\rho$ induces a bounded
Banach bundle map $\rho_0: \Xi_{\Omega\setminus \sigma(T)} \rightarrow
\Lambda_{\Delta\setminus T}$ (covering $\sigma|_{\Delta\setminus T}$).
Moreover, $\sigma$ is continuous on
$\Delta\setminus \FZ_{\rho,\sigma}$ where $\FZ_{\rho,\sigma} :=
\{\nu\in \Delta: \rho(e(\sigma(\nu))) = 0 {\rm\ for\ all\ } e\in E\}$.
\end{cor}
\begin{prf}
The first conclusion follows from Theorem \ref{thm:decomp}. To see
the second conclusion, we note that $\theta$ is separating and we
can apply Lemma \ref{lem-diff-base} (observe that $\FZ_{\rho,\sigma} =
\FZ_\theta$).
\end{prf}

\medskip

\begin{thm}\label{thm:biseparating}
Let $\Omega$ and $\Delta$ be two locally compact Hausdorff spaces,
and let $E$ be a full essential Banach
$C_0(\Omega)$-module (see Remark \ref{rem:FV}(b)) and $F$ be a full essential Banach
$C_0(\Delta)$-normed module. Suppose that $\theta: E\rightarrow F$
is a bijective $\mathbb{C}$-linear map (not assumed to be bounded)
such that it is \emph{biseparating} in the sense that both $\theta$
and $\theta^{-1}$ are separating.

\smnoind (a) There exists a homeomorphism $\sigma: \Delta
\rightarrow \Omega$ satisfying
$$\theta(e\cdot \varphi)\ =\ \theta(e)
\cdot \varphi\circ \sigma \qquad (e\in E; \varphi\in C_0(\Omega)).$$

\smnoind (b) There exists isolated points $\nu_1,...,\nu_n\in
\Delta$ such that the restriction of $\theta$ induces a Banach space
isomorphism $\theta_0: E_{\Omega_\theta} \rightarrow
F_{\Delta_\theta}$, where $\Delta_\theta := \Delta\setminus
\{\nu_1,...,\nu_n\}$ and $\Omega_\theta := \sigma(\Delta_\theta)$.
\end{thm}
\begin{prf}
(a) If $e\in E$ with $\ti e =0$, then $\theta(e) = \ti \theta(e) =
0$ (as $\theta$ is separating and $F$ is $C_0(\Delta)$-convex), which gives
$e = 0$ (as $\theta$ is injective).
Hence, one can regard $\widetilde{\theta^{-1}}
= \theta^{-1}$ as well.
The fullness of $E$ and $F$ as well as the
surjectivity of $\theta$ and $\theta^{-1}$ ensure that $\FZ_\theta =
\emptyset$ and $\FZ_{\theta^{-1}} = \emptyset$.
Therefore, by Lemma
\ref{lem-diff-base}, we have two continuous maps
$$\tau: \Omega \rightarrow \Delta_\infty \quad {\rm and} \quad \sigma: \Delta \rightarrow \Omega_\infty$$
such that $\theta^{-1}\left(I_{\tau(\omega)}^F\right)\subseteq I_\omega^E$
($\omega\in \Omega$) and $\theta\left(I_{\sigma(\nu)}^E\right) \subseteq
I_\nu^F$ ($\nu\in \Delta$). Consequently, for any $\nu\in \Delta_0 :
= \sigma^{-1}(\Omega)$ and $\omega\in \Omega_0:=\tau^{-1}(\Delta)$,
we have
$$\sigma(\tau(\omega)) = \omega \quad {\rm and} \quad \tau(\sigma(\nu)) = \nu$$
(because $I^E_{\sigma(\tau(\omega))} \subseteq I^E_\omega$,
$I^F_{\tau(\sigma(\nu))} \subseteq I^F_\nu$, and $E$ as well as $F$ are
full).
If there exists $\nu\in \Delta \setminus
\FN_{\theta,\sigma}$ ($\FN_{\theta,\sigma}$ as in Lemma
\ref{lem:C_0-lin}(b)) with $\sigma(\nu) = \infty$, then $F =
\theta\left(K_\infty^E\right) \subseteq K_\nu^F$, which contradicts the fullness
of $F$.
Thus,
$$\Delta\setminus \FN_{\theta,\sigma}\ \subseteq\ \Delta_0.$$
On the other hand, as $\Delta_0\cap \FN_{\theta,\sigma}$ is a finite
set (by Lemma \ref{lem:C_0-lin}(a)\&(b) and the fact that $\sigma$
is injective on $\Delta_0$) and is open in $\Delta$ (by Lemma
\ref{lem:auto-bdd}(a)), we see that $\Delta_0\cap
\FN_{\theta,\sigma}$ consists of isolated points of $\Delta$.
Thus, $\sigma(\Delta_0\cap \FN_{\theta,\sigma})$ consists of isolated
points of $\Omega_0$ (as $\sigma$ restricts to a homeomorphism
from $\Delta_0$ to $\Omega_0$).
We want to show that
$$\Delta_0\cap
\FN_{\theta,\sigma}\ =\ \emptyset.$$
Suppose on the contrary that there is $\nu\in \Delta_0\cap
\FN_{\theta,\sigma}$.
We know that $\sigma(\nu)$ ($\neq \infty$) is
a non-isolated point of $\Omega_\infty$ (by Lemma
\ref{lem:C_0-lin}(b)).
Therefore, there exists a net
$\{\omega_i\}_{i\in I}$ in $\Omega\setminus \{\sigma(\nu)\}$
converging to $\sigma(\nu)$.
If $\{i\in I: \omega_i\in \Omega_0\}$
is cofinal, then there is a net in $\Omega_0 \setminus
\{\sigma(\nu)\}$ converging to $\sigma(\nu)$, which contradicts $\sigma(\nu)$ being an isolated point in $\Omega_0$.
Otherwise, $\omega_i\in \tau^{-1}(\infty)$ eventually, which gives
the contradiction that $\nu = \infty$ (note that
$\tau(\omega_i)\rightarrow \nu$ as $\nu\in \Delta_0$). Consequently,
$$\Delta\setminus \FN_{\theta,\sigma}\ =\ \Delta_0.$$
Suppose that $\FN_{\theta, \sigma} \neq \emptyset$ and $\nu\in
\FN_{\theta, \sigma}$.
Since $\FN_{\theta, \sigma}$ is an open
subset of $\Delta$ (by Lemma \ref{lem:auto-bdd}(a)), there exists
$V\in \CN_{\Delta}(\nu)$ with $V\subseteq \FN_{\theta,
\sigma}$.
Take any $f\in F$ such that $f(\nu) \neq 0$ (by the fullness of
$F$) and $f$ vanishes outside $V$. Thus, $f\in I_\infty^F$ (as $V$
is compact) and so, $\theta^{-1}(f)(\omega) = 0$ for any $\omega \in
\tau^{-1}(\infty)$. On the other hand, for any $\omega\in \Omega_0$,
one has $\tau(\omega)\in \Delta_0$ and so, $f\in I_{\tau(\omega)}^F$
(as $f$ vanishes on the open set $\Delta_0$ containing
$\tau(\omega)$) which implies that $\theta^{-1}(f)(\omega) = 0$.
Hence $\theta^{-1}(f) = 0$ which contradicts the injectivity of
$\theta^{-1}$. Therefore, $\FN_{\theta,\sigma} = \emptyset$. Now,
part (a) follows from Lemma \ref{lem:C_0-lin}(c).

\smnoind (b) This follows directly from Theorem \ref{thm:decomp}(b).
\end{prf}

\medskip

One can apply the above to the case when $F$ is a full Hilbert
$C_0(\Delta)$-module.
Another direct application of Theorem \ref{thm:biseparating} is the following theorem which extends and enriches a result of Chan \cite{Ch90} (by removing the boundedness assumption on $\theta$), as well as results concerning the product bundle cases discussed in \cite{AJ03studia, GJW03}.
Notice that if $(\Omega,\{\Xi_x\}, E)$ is a continuous fields of Banach spaces over a locally compact Hausdorff space $\Omega$ (as defined in \cite{Di77, Fe61}), then $E$ is a full essential Banach $C_0(\Omega)$-normed module.

\medskip

\begin{thm}
Let $(\Omega,\{\Xi_x\}, E)$ and $(\Delta, \{\Lambda_y\}, F)$ be continuous fields of Banach spaces over locally compact Hausdorff spaces $\Omega$ and $\Delta$ respectively.
Let $\theta :  E \to  F$ be a bijective linear map such that both $\theta$ and its inverse $\theta^{-1}$ are separating.
Then there is a homeomorphism $\sigma: \Delta \to \Omega$ and a bijective linear operator $H_\nu: \Xi_{\sigma(\nu)}\to \Lambda_\nu$ such that
$$
\theta(f)(\nu)\ =\ H_\nu(f(\sigma(\nu))) \quad (f\in E, \nu\in \Delta).
$$
Moreover, at most finitely many $H_\nu$ are unbounded, and this can happen only when $\nu$ is an isolated point in $\Delta$.
In particular, if $\Omega$ (or $\Delta$) contains no isolated point, then $\theta$ is automatically bounded.
\end{thm}

\bigskip

\appendix\section{Bounded local linear maps are $A$-linear}\label{s:app}

\medskip

\begin{prop}\label{prop:local+bounded=modulemaps}
Let $A$ be a $C^*$-algebra, and let $\theta$ be a bounded linear map
from a Banach right $A$-modules $E$ into a Hilbert $A$-module $F$.
Then $\theta$ is a right $A$-module map if and only if $\theta$ is
\emph{local} (in the sense that $\theta(e)a = 0$ whenever $e\in E$
and $a\in A$ with $ea = 0$).
\end{prop}
\begin{proof}
Suppose $\theta$ is local. Observe, first of all, that $E^{**}$ and
$F^{**}$ are unital Banach $A^{**}$-modules, and the bidual map
$\theta^{**}: E^{**} \rightarrow F^{**}$ is a bounded weak*-weak*-continuous linear map.
Fix $x\in E$ and $a\in A_+$, and let
$$\Phi: C(\sigma(a))^{**} \rightarrow A^{**}$$
be the map induced by the canonical normal $*$-homomorphism $\Psi:
M(A)^{**} \rightarrow A^{**}$.
Pick $\alpha, \beta\in \mathbb{R}_+$
with $\alpha < \beta$, and define $p := \Phi
(\chi_{\sigma(a)\cap(\alpha, \beta)})$.  Let $\{f_n\}$ and
$\{g_n\}$ be two bounded sequences in $C(\sigma(a))_+$ such that
$f_n g_n =0$, as well as
$$
\quad f_n \uparrow \chi_{\sigma(a)\cap(\alpha, \beta)}
\quad {\rm and} \quad g_n \downarrow \chi_{\sigma(a)\setminus (\alpha, \beta)}\quad {\rm pointwisely}.
$$
Note that as $\Psi(A) \subseteq A$, we have $a_n := \Phi(f_n)\in A$,
and we can write $b_n := \Phi(g_n)$ as $c_n+\gamma_n 1$ (where
$c_n\in A$ and $\gamma_n\in \mathbb C$).
Fix $n\in \mathbb{N}$.
Since
$a_n$ and $c_n$ commute, there is a locally compact Hausdorff space
$\Omega$ with $C^*(a_n,c_n) \cong C_0(\Omega)$.
By considering
$b_n\in C(\Omega_\infty)_+ \cong C^*(1,a_n,c_n)_+$, one can find a net
$\{d_i\}_{i\in I}$ in $C_0(\Omega)_+\subseteq A_+$ such that
$d_i\leq b_n$ ($i\in I$) and $d_i\rightarrow b_n$ pointwisely.
As $0
\leq d_i \leq b_n$ and $a_nb_n = 0$ in $C(\Omega_\infty)$, one knows
that $a_nd_i = 0$.
Now, the relation $\theta(xa_n)d_i = 0$ and
$\theta(xd_i)a_n = 0$ imply that $\theta^{**}(xa_n)b_n = 0$ and
$\theta^{**}(xb_n)a_n = 0$.
Since the multiplication in the bidual
of the linking algebra of $F$ is jointly weak*-continuous on
bounded subsets, we see that $\theta^{**}(xp)(1-p) = 0$ and
$\theta^{**}(x(1-p))p = 0$, which implies that $\theta^{**}(xp) =
\theta^{**}(x) p$.
Finally, there exists $r_k\in \mathbb{R}$ and
$\alpha_k, \beta_k\in \mathbb{R}_+$ such that $\alpha_k \leq
\beta_k$ and
$$\sup_{t\in \sigma(a)} \bl | a(t) - \sum_{k=1}^M r_k \chi_{\sigma(a)\cap(\alpha_k, \beta_k)}(t)\br | \rightarrow 0.$$
Thus, by the weak*-continuity again, we get $\theta^{**}(xa) =
\theta^{**}(x)a$ as required.
\end{proof}


\bigskip

\begin{thebibliography}{99}

\bibitem{Ab83}
Yu. A. Abramovich, \emph{Multiplicative representation of the
operators preserving disjointness}, Indag.\ Math.\ \textbf{45}
(1983), 265--279.




\bibitem{AN80}
E. Albrecht and M. M. Neumann, \emph{Automatic continuity of
generalized local linear operators}, Manuscripta Math.\ \textbf{32}
(1980), 263--294.


\bibitem{Ar04}
J. Araujo, \emph{Linear biseparating maps between spaces of
vector-valued differentiable functions and automatic continuity},
Adv. Math. \textbf{187} (2004), no.~2, 488--520.


\bibitem{AJ03studia}
J. Araujo and K. Jarosz, \emph{Automatic continuity of biseparating
maps}, Studia Math. \textbf{155} (2003), no.~3, 231--239.


\bibitem{Ar83}
W. Arendt, \emph{Spectral properties of Lamperti operators}, Indiana
Univ.\ Math.\ J. \textbf{32} (1983), 199--215.

\bibitem{AT05}
W. Arendt and S. Thomaschewski, \emph{Local operators and forms},
Positivity \textbf{9} (2005), 357--367.

\bibitem{BNT88}
E. Beckenstein, L. Narici, and A. R. Todd, \emph{Automatic
continuity of linear maps on spaces of continuous functions},
Manuscripta Math.\ \textbf{62} (1988), 257--275.


\bibitem{Ch90}
J. T. Chan, \emph{Operators with the disjoint support property}, {J.
Operator Theory} \textbf{24} (1990), 383--391.


\bibitem{Di77}
J. Dixmier, \emph{C*-algebras}, North-Holland publishing company,
Amsterdam--New York--Oxford, 1977.


\bibitem{DG83}
M. J. Dupr\'{e} and R. M. Gillette, \emph{Banach bundles, Banach
modules and automorphisms of $C\sp{*} $-algebras}, Research Notes in
Mathematics 92, Pitman (1983).

\bibitem{Fe61}
J. M. G. Fell, \emph{The structure of algebras of operator fields},
{Acta Math.}, \textbf{106} (1961),
 233--280.

\bibitem{FH94}
J. J. Font and S. Hern\'{a}ndez, \emph{On separating maps between
locally compact spaces}, Arch.\ Math.\ (Basel) \textbf{63} (1994),
158--165.

\bibitem{GJW03}
H.-L. Gau, J.-S. Jeang and N.-C. Wong, \emph{Biseparating linear
maps between continuous vector-valued function spaces}, {J.
Australian Math. Soc., Series A}, \textbf{74} (2003), no.~1,
101--111.

\bibitem{Ja90}
K. Jarosz, \emph{Automatic continuity of separating linear
isomorphisms}, Canad.\ Math.\ Bull.\ \textbf{33} (1990), 139--144.


\bibitem{JW96} J.-S. Jeang and N.-C. Wong,
  Weighted composition operators of $C_0(X)$'s, \emph{J.
    Math.\ Anal.\ Appl.} \textbf{201} (1996), 981--993.



\bibitem{KN01} R. Kantrowitz and M. M. Neumann, Disjointness preserving and local
operators on algebras of differentiable functions, Glasg. Math. J.
43 (2001), 295-309.

\bibitem{Na85}
R. Narasimhan, \emph{Analysis on real and complex manifolds},
Advanced Studies in Pure Mathematics \textbf{1}, North-Holland
Publishing Co., Amsterdam 1968

\bibitem{Pa84}
 B. de Pagter,
 \emph{A note on disjointness preserving operators},
 Proc.\ Amer.\ Math.\ Soc.\ \textbf{90} (1984), 543--550.


\bibitem{Pe60}
J. Peetre, \emph{R\'{e}ctification \`{a} l'article ``Une
caract\'{e}risation abstraite des op\'{e}rateurs
diff\'{e}rentiels''}, Math.\ Scand.\ \textbf{8} (1960), 116--120.


\bibitem{Sw62}
R. G. Swan, \emph{Vector bundles and projective modules}, Trans.\
Amer.\ Math.\ Soc.\ \textbf{105} (1962), 264--277.

\end{thebibliography}
\end{document}